\documentclass[a4paper,reqno]{amsart}

\usepackage{amssymb}
\usepackage{latexsym}
\usepackage{amsmath}
\usepackage[mathcal]{euscript}
\usepackage{bbm}
\usepackage{tikz}
\usepackage{xcolor}

   \def\cH{{\mathcal H}}   
      \def\cL{{\mathcal L}}

\def\cS{{\mathcal S}}

\def\cal H{{\mathcal H}}

\def\Q{\mathbb{Q}}
\def\R{\mathbb{R}}
\def\C{\mathbb{C}}

\def\N{\mathbb{N}}

\def\ran{{\text{\rm ran\,}}}
\def\dom{{\text{\rm dom\,}}}

\def\phi{\varphi}

\DeclareMathOperator{\cyc}{cyc}

\DeclareMathOperator{\Res}{Res}

\DeclareMathOperator{\spann}{span}
\DeclareMathOperator{\rank}{rank}

\newtheorem{theorem}{Theorem}[section]
\newtheorem*{thm*}{Theorem}
\newtheorem{proposition}[theorem]{Proposition}
\newtheorem{corollary}[theorem]{Corollary}
\newtheorem{lemma}[theorem]{Lemma}

\theoremstyle{definition}
\newtheorem{definition}[theorem]{Definition}
\newtheorem{example}[theorem]{Example}

\newtheorem{remark}[theorem]{Remark}
\newtheorem*{ack}{Acknowledgements}

\numberwithin{equation}{section}

\title[Recovering quantum graph spectrum from vertex data]{Recovering quantum graph spectrum from vertex data}

\author{Jonathan Rohleder}

\address{Institut f\"ur Numerische Mathematik \\
Technische Universit\"at Graz \\
Steyrergasse 30\\
A-8010 Graz\\
Austria}
\email{rohleder@tugraz.at}

\begin{document}

\begin{abstract}
We study the question to what extent spectral information of a Schr\"odinger operator on a finite, compact metric graph subject to standard or $\delta$-type matching conditions can be recovered from a corresponding Titchmarsh--Weyl function on the boundary of the graph. In contrast to the case of ordinary or partial differential operators, the knowledge of the Titchmarsh--Weyl function is in general not sufficient for recovering the complete spectrum of the operator (or the potentials on the edges). However, it is shown that those eigenvalues with sufficiently high (depending on the cyclomatic number of the graph) multiplicities can be recovered. Moreover, we prove that under certain additional conditions the Titchmarsh--Weyl function contains even the full spectral information.
\end{abstract}

\maketitle

\section{Introduction}

Quantum graphs, i.e.\ differential operators on metric graphs, provide mathematical models for a wide range of problems in physics and engineering such as, e.g., quantum wires, photonic crystals, or thin waveguides; see the recent monograph~\cite{BK13} and the references therein. Therefore their investigation has been a very active field of research in recent time. Amongst others, there is a strong interest in inverse problems for quantum graphs, see, e.g,~\cite{AK08,BW05,CIM98,D13,GS01,KN05,P07} for a small selection.

It is a question of particular interest how much information on a Schr\"odinger operator on a finite, compact metric graph~$G$ with an integrable potential $q : G \to \R$ can be recovered from an appropriate Titchmarsh--Weyl function corresponding to the Schr\"odinger equation on the graph. The Titchmarsh--Weyl matrix function~$M_B$ for a given set of boundary vertices (i.e. vertices of degree one) $B = \{v_1, \dots, v_m\} \subset \partial G$ of the graph can be defined by the relation
\begin{align}\label{eq:MIntro}
 M_B (\mu) \begin{pmatrix} \partial_\nu f (v_1) \\ \vdots \\ \partial_\nu f (v_m) \end{pmatrix} = \begin{pmatrix} f (v_1) \\ \vdots \\ f (v_m) \end{pmatrix},
\end{align}
where $f$ is any square-integrable function on~$G$ with $- f'' + q f = \mu f$ on each edge such that~$f$ is continuous at each vertex and satisfies a standard matching condition at every vertex which does not belong to $B$; here $\partial_\nu f (v)$ denotes the derivative of $f$ at~$v$ in the direction pointing outwards; see Section~\ref{sec:prelim1} for all details. The $|B| \times |B|$-matrix $M_B (\mu)$ is well-defined for each $\mu$ outside the purely discrete spectrum of the selfadjoint Schr\"odinger operator~$A$ in $L^2 (G)$ acting as $- \frac{d^2}{d x^2} + q$ on each edge and equipped with standard (also called Kirchhoff) matching conditions at all vertices of the graph. The matrix function~$M_B$ appears as a natural starting point of the inverse problem since it can be measured in boundary experiments; cf.~\cite{K11}.

In contrast to the case of a Schr\"odinger operator on an interval or on a domain in~$\R^n$, where the eigenvalues coincide with the poles of the corresponding Titchmarsh--Weyl function and its multiplicities equal the rank of the corresponding residue, in general the function~$M_B$ does not contain the complete information of the Schr\"odinger equation on the graph. It does not determine the potential~$q$ or the spectrum of the selfadjoint operator $A$ in~$L^2 (G)$ uniquely, except for very special cases such as, e.g., if $G$ is a tree or if additional information is provided; cf.~\cite{AK08,BW05,EK11,FY07,K09,K11}. 

The aim of this paper is to study which parts of the spectrum of~$A$ can be recovered, nevertheless, from the knowledge of the Titchmarsh--Weyl function~$M_B$, depending on the choice of the vertex set~$B$ on which the boundary data is assumed to be available. Besides discussing several counterexamples, we prove two main positive results. In Theorem~\ref{thm:generalBoundary} we show that it is possible to recover from the Titchmarsh--Weyl function all eigenvalues with sufficiently large multiplicities. More precisely, we show that all those eigenvalues of~$A$ appear as poles of the Titchmarsh--Weyl function whose multiplicities are strictly larger than $2 \cyc (G) + |\partial G| - |B| - 1$, where $\cyc (G)$ is the cyclomatic number of~$G$; see Section~\ref{sec:boundary} below for the details. In addition, we provide estimates which relate the 
multiplicity of a given eigenvalue~$\lambda$ of~$A$ to the rank of the 
residue of~$M_B$ at~$\lambda$. Moreover, in two corollaries we point out how the result reads for special choices of the graph or of the vertex set~$B$. We remark that, in contrast to the Schr\"odinger equation on a single interval, eigenvalues with high multiplicities do appear in many cases.

The second main result treats the case where the Titchmarsh--Weyl function is given for a larger vertex set~$B$ which includes interior vertices; for an interior vertex~$v$ the term $\partial_\nu f (v)$ in~\eqref{eq:MIntro} has to be understood as the sum of the derivatives at~$v$ of the restrictions of~$f$ to the edges attached to~$v$. 
In Theorem~\ref{thm:resonanceSingle} we assume that the set~$B$ contains all boundary vertices as well as, roughly speaking, the vertices which belong to the cycles of the graph; this is specified in a precise manner in Section~\ref{sec:large} below. Under these assumptions we show that it is possible to recover from~$M_B$ all eigenvalues of~$A$ and its multiplicities, provided that, additionally, a non-resonance condition on the quantum graph is satisfied. For the special case of a graph with one cycle this condition appeared earlier in~\cite{K09}; cf.~also~\cite{P07}. If the potential is absent, that is, $A$ acts as the Laplacian on~$G$, the result reads as follows: If each cycle in~$G$ contains two edges with rationally independent lengths then~$M_B$ carries the full spectral information of~$A$. A condition of this type seems natural; it is required, for example, in order to have a unique solution in the inverse problem of recovering the geometry of a metric graph from the 
spectrum of the corresponding Laplacian, see~\cite{KN05,N07}.

The effect that the Titchmarsh--Weyl function~$M_B$ in general does not detect all eigenvalues is due to the existence of so-called scars, i.e., eigenfunctions whose support does not contain the boundary of the graph (or which vanish on the observed vertex set~$B$, respectively); cf.~the various examples in Section~\ref{sec:boundary} and~\ref{sec:large}. For the case of a quantum graph without potentials it was observed in~\cite{SK03} that scars can only exist if the graph possesses periodic orbits with commensurate edges; this is in accordance with Corollary~\ref{cor:LaplaceEV} below. A characterization of scars by means of semiclassical measures was provided recently in~\cite{C14}. Scars are closely related to the corresponding scattering matrix which is obtained when attaching infinite leads to the compact graph; they appear if resonances, i.e.\ poles of the scattering matrix, hit the real axis. Such resonances cannot be noticed in a scattering experiment where only the transmission and reflection are 
measured, see, e.g.,~\cite{WS13}. For related work we refer the reader to~\cite{BKW04,E14,EL10,
GSS13,
KS04,TM01} and the references therein. Taking into account that the Titchmarsh--Weyl matrix and the scattering matrix are basically related to each other by a Cayley transform, see, e.g.,~\cite[Section~5.4]{BK13}, this is in accordance with the fact that those eigenvalues of~$A$ may fail to materialize as poles of the function~$M_B$. Therefore the main issue in the proofs of our main results is to estimate from above the number of linearly independent scars corresponding to an eigenvalue, or to exclude scars totally, respectively.

The results of this paper are not restricted to standard matching conditions. In fact, everything is carried out for the more general case of matching conditions of $\delta$-type, where the standard conditions are contained as the special case that the strength of the~$\delta$-interaction is zero at each vertex. Moreover, we remark that the results do also remain valid for $\delta'$- and further local vertex conditions, but we do not go into these details. 

Let us mention that there are definitions of a Titchmarsh--Weyl function for quantum graphs which differ substantially from ours. For certain purposes it can be convenient to use the orthogonal sum of the Titchmarsh--Weyl functions of all the single edges instead of the function defined in~\eqref{eq:MIntro}. This leads to a function whose values are matrices of the much larger size $2 r \times 2 r$, where~$r$ is the number of edges of the graph. This larger matrix function does always contain the complete spectral information, see, e.g.,~\cite{BL10,CW09}, but the knowledge of that functions requires full information on the boundary values of the solutions on each single edge, which may be unavailable.

\begin{ack}
The author is grateful to B.~Malcolm~Brown for helpful comments on an early draft of this paper and to Jussi Behrndt, Hannes Gernandt, Christian K\"uhn, Pavel Kurasov, Daniel Lenz, and Olaf Post for stimulating discussions on related subjects. Furthermore, the author wishes to thank the anonymous referees for their valuable hints and suggestions.
\end{ack}

\section{Quantum graphs with standard or $\delta$-matching conditions and Titchmarsh--Weyl functions}\label{sec:prelim1}

In this paragraph we fix some notation and recall basic facts on selfadjoint Schr\"odinger operators on finite, compact metric graphs with standard and $\delta$-type matching conditions. For further details we refer the reader to the recent monographs~\cite{BK13,P12} and to~\cite{KS99,KS06,K08} and their references. 

A finite, compact metric graph~$G$ is a collection of finitely many compact intervals $I_j = [0, L_j]$, with $L_j > 0$, $j = 1, \dots, r$, with an equivalence relation on the set of endpoints of these intervals. The intervals $I_1, \dots, I_r$ are called edges and the equivalence classes $v_1, \dots, v_s$ of endpoints are called vertices. In the following we write $o (I_j)$ for the vertex from which the edge $I_j$ originates, that is, the vertex which corresponds to the zero endpoint of~$I_j$, and $t (I_j)$ for the vertex at which~$I_j$ terminates, that is, which corresponds to the endpoint $L_j$, $j = 1, \dots, r$. We say that $I_j$ and $I_k$ are adjacent edges if they are attached to a joint vertex. Moreover, we define the degree $\deg (v_i)$ of a vertex~$v_i$, $i = 1, \dots, s$, to be the cardinality of the equivalence class~$v_i$, that is, the number of edges attached to~$v_i$. Note that with this definition a loop, i.e. an edge~$I_j$ with~$o (I_j) = t (I_j)$, counts twice for the calculation of 
the degree. We denote by~$\partial G$ the boundary of~$G$, that is, the set of vertices of~$G$ with degree one.

A function~$f : G \to \C$ is understood as collection of $r$ functions $f_j : I_j \to \C$, $j = 1, \dots, r$. Accordingly we set
\begin{align*}
 L^2 (G) = \bigoplus_{j = 1}^r L^2 (I_j),
\end{align*}
equipped with the standard norm and inner product, where we denote the latter by~$(\cdot, \cdot)$. Moreover, we say that a function $f : G \to \C$ is continuous on~$G$ if $f_j : I_j \to \C$ is continuous for~$j = 1, \dots, r$ and for each two adjacent edges $I_j$ and $I_k$ with joint vertex~$v_i$, $i \in \{1, \dots, s\}$, the values of $f_j$ and $f_k$ at~$v_i$ coincide. If~$f$ is continuous on~$G$ we just write $f (v_i)$ for the value of an arbitrary component of $f$ at~$v_i$.

Let $q_j \in L^1 (I_j)$ be real-valued functions, $j = 1, \dots, r$, and consider the Schr\"odin\-ger differential expression on $G$ formally given by
\begin{align*}
 (\cL f)_j= - f_j'' + q_j f_j \quad \text{on} \quad I_j, \qquad j = 1, \dots, r.
\end{align*}
The selfadjoint operators under consideration associated with~$\cL$ satisfy vertex conditions of $\delta$-type of the form
\begin{align}\label{eq:MatchCond}
 f~\text{is~continuous~on}~G \quad \text{and} \quad \partial_\nu f (v_i) - \alpha_i f (v_i) = 0, \quad i = 1, \dots, s,
\end{align}
where $\alpha = (\alpha_1, \dots, \alpha_s)$ is a vector with real entries, to be considered as the strength of the~$\delta$-interaction, and
\begin{align*}
 \partial_\nu f (v_i) = \sum_{t (I_j) = v_i} f_j' (L_j) - \sum_{o (I_j) = v_i} f_j' (0), \quad i = 1, \dots, s;
\end{align*}
if~$\deg (v_i) = 1$ then this expression reduces to one summand. In the particular case $\alpha = 0$,~\eqref{eq:MatchCond} equals the usual standard (or Kirchhoff) matching conditions. For brevity let us set
\begin{align*}
 AC^2 (G) := \left\{ f : G \to \C : f_j, f_j'~\text{absolutely~continuous~on}~(0, L_j), j = 1, \dots, r \right\}.
\end{align*}
The operators in~$L^2 (G)$ corresponding to the vertex conditions~\eqref{eq:MatchCond} are defined by
\begin{align}\label{eq:Aalpha}
\begin{array}[]{rcl}
 A_\alpha f & \!\!\!\! = \!\!\!\! & \cL f, \vspace{2mm} \\
 \dom A_\alpha & \!\!\!\! = \!\!\!\! & \left\{ f \in L^2 (G) \cap AC^2 (G) : \cL f \in L^2 (G), f~\text{satisfies}~\eqref{eq:MatchCond} \right\}.
\end{array}
\end{align}
We collect basic properties of~$A_\alpha$ in the following proposition. They can be shown by using standard techniques; cf.~\cite[Chapter~1]{BK13}. For properties of analytic and meromorphic operator functions we refer the reader to the short exposition given in the appendix of this paper.

\begin{proposition}\label{prop:operator}
For each~$\alpha \in \R^s$ the following assertions hold.
\begin{enumerate}
 \item The operator $A_\alpha$ in~\eqref{eq:Aalpha} is selfadjoint in~$L^2 (G)$ and semibounded below (with a lower bound depending on~$\alpha$ and the potentials~$q_j$, $j = 1, \dots, r$).
 \item The resolvent~$(A_\alpha - \mu)^{-1}$ is a compact operator in~$L^2 (G)$ for each~$\mu \in \C \setminus \sigma (A_\alpha)$, where~$\sigma (A_\alpha)$ denotes the spectrum of~$A_\alpha$. In particular, $\sigma (A_\alpha)$ consists of isolated eigenvalues with finite multiplicities, which accumulate to~$+ \infty$.
 \item The resolvent $\mu \mapsto (A_\alpha - \mu)^{-1}$ is a meromorphic operator function with poles of order one precisely at the eigenvalues of~$A_\alpha$.
\end{enumerate}
\end{proposition}

We remark that, in contrast to the case of a Schr\"odinger operator on a single interval, the operators~$A_\alpha$ can have eigenvalues with high multiplicities, depending on the geometry of the graph and the choice of the potentials, see, e.g.,~\cite{KP11}.

Let us next come to the definition of the Titchmarsh--Weyl function corresponding to the operator $A_\alpha$ in~\eqref{eq:Aalpha}. For this we make use of the following lemma.

\begin{lemma}\label{lem:BVPunique}
Let $\alpha \in \R^s$ and let $A_\alpha$ be the selfadjoint operator in~\eqref{eq:Aalpha}. Then for each $\mu \in \C \setminus \sigma (A_\alpha)$ and each $h = (h_1, \dots, h_s) \in \C^s$ there exists a unique solution $f$ of the differential equation $\cL f = \mu f$ such that~$f \in L^2 (G) \cap AC^2 (G)$, $f$ is continuous on~$G$, and
\begin{align}\label{eq:inhomBC}
 \partial_\nu f (v_i) - \alpha_i f (v_i) = h_i, \quad i = 1, \dots, s.
\end{align}
\end{lemma}

\begin{proof}
Let $\mu$ and~$h$ be fixed as in the statement of the lemma. For the uniqueness, assume that $f_1, f_2 \in L^2 (G) \cap AC^2 (G)$ are two continuous functions on~$G$ satisfying~$\cL f_j = \mu f_j$, $j = 1, 2$, and the vertex conditions~\eqref{eq:inhomBC}. Then $f := f_1 - f_2$ belongs to $\ker (A_\alpha - \mu)$, which, together with~$\mu \in \C \setminus \sigma (A_\alpha)$, implies $f = 0$, that is, $f_1 = f_2$.

For the existence let $g \in L^2 (G) \cap AC^2 (G)$ be an arbitrary continuous function on~$G$ with~$\cL g \in L^2 (G)$ and
\begin{align*}
 \partial_\nu g (v_i) - \alpha_i g (v_i) = h_i, \quad l = 1, \dots, s,
\end{align*}
and define $f := g - (A_\alpha - \mu)^{-1} (\cL - \mu) g \in L^2 (G) \cap AC^2 (G)$. Since $(A_\alpha - \mu)^{-1}$ maps into the domain of~$A_\alpha$ we obtain that $f$ is continuous on~$G$, $(\cL - \mu) f = 0$, and $\partial_\nu f (v_i) - \alpha_i f (v_i) = \partial_\nu g (v_i) - \alpha_i g (v_i)$; the latter yields~\eqref{eq:inhomBC}.
\end{proof}

It is a consequence of Lemma~\ref{lem:BVPunique} that the following definition makes sense. We introduce the (matrix-valued) Titchmarsh--Weyl function acting on an arbitrary subset~$B$ of the vertex set of~$G$ in the following way.

\begin{definition}\label{def:TW}
Let $B := \{v_{i_1}, \dots, v_{i_m}\}$ be a nonempty subset of the set of vertices of~$G$. Moreover, let $\alpha \in \R^s$, let $A_\alpha$ be defined as in~\eqref{eq:Aalpha}, and let $\mu \in \C \setminus \sigma (A_\alpha)$. For $l = i_1, \dots, i_m$ let $f^{(l)} \in L^2 (G) \cap AC^2 (G)$ such that~$\cL f^{(l)} = \mu f^{(l)}$, $f^{(l)}$ is continuous on~$G$, $\partial_\nu f^{(l)} (v_l) - \alpha_i f^{(l)} (v_l) = 1$, and
\begin{align*}
  \partial_\nu f^{(l)} (v_i) - \alpha_i f^{(l)} (v_i) = 0, \quad i = 1, \dots, l - 1, l + 1, \dots, s.
\end{align*}
The {\em Titchmarsh--Weyl matrix} $M_{B, \alpha} (\mu) \in \C^{m \times m}$ is the matrix with the entries
\begin{align*}
 (M_{B, \alpha} (\mu))_{k, l} = f^{(l)} (v_{i_k}), \quad k, l = 1, \dots, m.
\end{align*}
The matrix function $\mu \mapsto M_{B, \alpha} (\mu)$ is called {\em Titchmarsh--Weyl function}.
\end{definition}

Equivalently, the Titchmarsh--Weyl function $M_{B, \alpha}$ satisfies
\begin{align}\label{eq:ND}
 M_{B, \alpha} (\mu) \begin{pmatrix} \partial_\nu f (v_{i_1}) - \alpha_{i_1} f (v_{i_1}) \\ \vdots \\ \partial_\nu f (v_{i_m}) - \alpha_{i_m} f (v_{i_m}) \end{pmatrix} = \begin{pmatrix} f (v_{i_1}) \\ \vdots \\ f (v_{i_m}) \end{pmatrix}, \quad \mu \in \C \setminus \sigma (A_\alpha),
\end{align}
where $f$ is any solution in $L^2 (G) \cap AC^2 (G)$ of $\cL f = \mu f$ such that $f$ is continuous on~$G$ and satisfies $\partial_\nu f (v_i) - \alpha_i f (v_i) = 0$ whenever $i \neq i_1, \dots, i_m$. If $\alpha = 0$ then $M_{B, \alpha} (\mu)$ is a Neumann-to-Dirichlet map for~$\cL - \mu$ on~$G$. We remark that often the Titchmarsh--Weyl function is considered to be the function $\mu \mapsto (M_{B, \alpha} (\mu))^{-1}$, which, clearly, carries the same spectral information as our function~$M_{B, \alpha}$. However, in order to formulate our main results the latter will turn out to be more convenient.

In the following proposition we collect further properties of~$M_{B, \alpha}$ and establish its basic connection to the eigenvalues of~$A_\alpha$. Here and further on we write $\Res_\lambda M_{B, \alpha}$ for the residue of the (meromorphic) matrix function $M_{B, \alpha}$ at some~$\lambda \in \R$, with the reasonable convention $\Res_\lambda M_{B, \alpha} = 0$ if $M_{B, \alpha}$ is either analytic at~$\lambda$ (i.e. $\lambda \in \R \setminus \sigma (A_\alpha)$) or can be continued analytically into~$\lambda$; cf.~Appendix.

\begin{proposition}\label{prop:entscheidend}
Let $B = \{v_{i_1}, \dots, v_{i_m}\}$ be a nonempty set of vertices of~$G$. Moreover, let $\alpha \in \R^s$ and let~$A_\alpha$ be defined in~\eqref{eq:Aalpha}. Then the matrix function $\mu \mapsto M_{B, \alpha} (\mu)$ in Definition~\ref{def:TW} is analytic on~$\C \setminus \sigma (A_\alpha)$ and each (non-removable) singularity of~$M_{B, \alpha}$ is a pole of order one. Moreover, for each~$\lambda \in \R$ the linear mapping
\begin{align}\label{eq:gammaLambda}
 \gamma_\lambda : \ker (A_\alpha - \lambda) \to \C^m, \quad f \mapsto \begin{pmatrix} f (v_{i_1}) \\ \vdots \\ f (v_{i_m}) \end{pmatrix},
\end{align}
satisfies $\ran \gamma_\lambda = \ran \Res_\lambda M_{B, \alpha}$. In particular,
\begin{align*}
 \dim \ker (A_\alpha - \lambda) = \rank \Res_\lambda M_{B, \alpha} + \dim \ker \gamma_\lambda.
\end{align*}
\end{proposition}

\begin{proof}
{\bf Step 1.} In this first step we establish the identity
\begin{align}\label{eq:MResolvent}
 M_{B, \alpha} (\mu) = M_{B, \alpha} (\zeta)^* + (\mu - \overline{\zeta}) \cS_\zeta^* \left( I + (\mu - \zeta) (A_\alpha - \mu)^{-1} \right) \cS_\zeta, \quad \zeta, \mu \in \C \setminus \sigma (A_\alpha),
\end{align}
where $\cS_\zeta : \C^m \to L^2 (G)$ is the solution operator which is defined by the identity
\begin{align*}
 \cS_\zeta \begin{pmatrix} \partial_\nu f (v_{i_1}) - \alpha_{i_1} f (v_{i_1}) \\ \vdots \\ \partial_\nu f (v_{i_m}) - \alpha_{i_m} f (v_{i_m}) \end{pmatrix} = f
\end{align*}
for each solution $f$ in $L^2 (G) \cap AC^2 (G)$ of $\cL f = \zeta f$ such that $f$ is continuous on~$G$ and satisfies $\partial_\nu f (v_i) - \alpha_i f (v_i) = 0$ whenever $i \neq i_1, \dots, i_m$; note that $\cS_\zeta$ is well-defined for each~$\zeta \in \C \setminus \sigma (A_\alpha)$ by Lemma~\ref{lem:BVPunique} and that~$\cS_\zeta$ is related to the Titchmarsh--Weyl matrix $M_{B, \alpha} (\zeta)$ via
\begin{align*}
 M_{B, \alpha} (\zeta) h = \begin{pmatrix}
                            (\cS_\zeta h) (v_{i_1}) \\ \vdots \\ (\cS_\zeta h) (v_{i_m})
                           \end{pmatrix}, \quad h \in \C^m;
\end{align*}
cf.~\eqref{eq:ND}. In order to verify~\eqref{eq:MResolvent}, let $\zeta, \mu \in \C \setminus \sigma (A_\alpha)$ and~$h, k \in \C^m$. First of all, define a function
\begin{align*}
 f_{(h)} = \big( I + (\mu - \zeta) (A_\alpha - \mu)^{-1} \big) \cS_\zeta h.
\end{align*}
Observe that $f_{(h)}$ belongs to $L^2 (G) \cap AC^2 (G)$ with
\begin{align*}
 (\cL - \mu) f_{(h)} & = (\cL - \mu + \mu - \zeta) \cS_\zeta h  = 0,
\end{align*}
that~$f_{(h)}$ is continuous on~$G$, and that $\partial_\nu f_{(h)} (v_{i_l}) - \alpha_{i_l} f_{(h)} (v_{i_l}) = h_l$, $l = 1, \dots, m$, and $\partial_\nu f_{(h)} (v_i) - \alpha_i f_{(h)} (v_i) = 0$ whenever $i \neq i_1, \dots, i_m$. Since the function~$\cS_\mu h$ has the same properties, the uniqueness assertion in Lemma~\ref{lem:BVPunique} implies $f_{(h)} = \cS_\mu h$. Hence,
\begin{align}\label{eq:schlimmSchlimm}
 (\mu - \overline \zeta) \left( \left[ I + (\mu - \zeta) (A_\alpha - \mu)^{-1} \right] \cS_\zeta h, \cS_\zeta k \right) & = (\mu - \overline \zeta) (\cS_\mu h, \cS_\zeta k) \nonumber \\
 & = (\cL \cS_\mu h, \cS_\zeta k) - (\cS_\mu h, \cL \cS_\zeta k).
\end{align}
Integration by parts yields for $f = \cS_\mu h$ and $g = \cS_\zeta k$
\begin{align}\label{eq:weiterGehts}
 (\cL \cS_\mu h, \cS_\zeta k) - (\cS_\mu h, \cL \cS_\zeta k) & = - \sum_{j = 1}^r \int_0^{L_j} f_j'' \overline{g_j} d x + \sum_{j = 1}^r \int_0^{L_j} f_j \overline{g_j''} d x \nonumber \\
 & = \sum_{j = 1}^r \Big( - f_j' (L_j) \overline{g_j (L_j)} + f_j' (0) \overline{g_j (0)} \nonumber \\
 & \quad + f_j (L_j) \overline{g_j' (L_j)} - f_j (0) \overline{g_j' (0)} \Big).
\end{align}
On the other hand, taking into account that $\partial_\nu f (v_i) - \alpha_i f (v_i) = 0$ and $\partial_\nu g (v_i) - \alpha_i g (v_i) = 0$ for each $i \neq i_1, \dots, i_m$, we have
\begin{align}\label{eq:malSehen}
 (M_{B,\alpha} (\mu) & h, k)_{\C^m} - (h, M_{B, \alpha} (\zeta) k )_{\C^m} \nonumber \\
 & = \sum_{i = 1}^s f (v_i) \overline{(\partial_\nu g (v_i) - \alpha_i g (v_i))} - \sum_{i = 1}^s (\partial_\nu f (v_i) - \alpha_i f (v_i)) \overline{g (v_i)} \nonumber \\
 & = \sum_{j = 1}^r \Big(f_j (L_j) \overline{g_j' (L_j)} - f_j (0) \overline{g_j' (0)} - f_j' (L_j) \overline{g_j (L_j)} + f_j' (0) \overline{g_j (0)} \Big),
\end{align}
where~$(\cdot, \cdot)_{\C^m}$ denotes the inner product in~$\C^m$. From this and the equations~\eqref{eq:schlimmSchlimm} and~\eqref{eq:weiterGehts} we conclude
\begin{align*}
 (\mu - \overline \zeta) \left( \left[ I + (\mu - \zeta) (A_\alpha - \mu)^{-1} \right] \cS_\zeta h, \cS_\zeta k \right) = (M_{B,\alpha} (\mu) h, k)_{\C^m} - (h, M_{B, \alpha} (\zeta) k )_{\C^m}.
\end{align*}
As $h, k \in \C^m$ were chosen arbitrarily the identity~\eqref{eq:MResolvent} follows.

{\bf Step 2.} In this second step we verify the assertions of the proposition. Recall that the operator function $\mu \mapsto (A_\alpha - \mu)^{-1}$ is analytic on~$\C \setminus \sigma (A_\alpha)$ and its only singularities are poles of order one; cf.~Appendix. Thus by means of~\eqref{eq:MResolvent} the same holds for the matrix function~$M_{B, \alpha}$. It remains to show the identity $\ran \gamma_\lambda = \ran \Res_\lambda M_{B, \alpha}$ for~$\lambda \in \R$. For this let~$\lambda \in \R$ and let~$P_\lambda$ be the orthogonal projection in~$L^2 (G)$ onto~$\ker (A_\alpha - \lambda)$. For $f \in L^2 (G)$,~$\zeta \in \C \setminus \sigma (A_\alpha)$, and $k \in \C^m$ we compute
\begin{align*}
 (\lambda - \overline \zeta) (\cS_\zeta^* P_\lambda f, k)_{\C^m} & = (\lambda P_\lambda f, \cS_\zeta k) - (P_\lambda f, \zeta \cS_\zeta k) \\
 & = (\cL P_\lambda f, \cS_\zeta k) - (P_\lambda f, \cL \cS_\zeta k) = (\gamma_\lambda P_\lambda f, k)_{\C^m},
\end{align*}
where we have used an integration by parts as in the equations~\eqref{eq:weiterGehts} and~\eqref{eq:malSehen} above and the facts that $\partial_\nu (P_\lambda f) (v_i) - \alpha_i (P_\lambda f) (v_i) = 0$ for $i = 1, \dots, s$ and~$\partial_\nu (\cS_\zeta k) (v_i) - \alpha_i (\cS_\zeta k) (v_i) = 0$ for $i \neq i_1, \dots, i_m$. Thus 
\begin{align}\label{eq:Projection}
 (\lambda - \overline \zeta) \cS_\zeta^* P_\lambda = \gamma_\lambda P_\lambda. 
\end{align}
With this knowledge and the relation~\eqref{eq:projection} in the appendix we obtain from~\eqref{eq:MResolvent}
\begin{align}\label{eq:schonEinAnfang}
 \Res_\lambda M_{B, \alpha} = - (\lambda - \overline \zeta) (\lambda - \zeta) \cS_\zeta^* P_\lambda \cS_\zeta = - (\lambda - \zeta) \gamma_\lambda P_\lambda \cS_\zeta.
\end{align}
The assertion on~$\ran \gamma_\lambda$ follows from~\eqref{eq:schonEinAnfang} if we can verify
\begin{align}\label{eq:nunMalLos}
 \{ \gamma_\lambda P_\lambda \cS_\zeta h : h \in \C^m \} = \ran \gamma_\lambda
\end{align}
for each $\zeta \in \C \setminus \sigma (A_\alpha)$. Indeed, let~$\zeta \in \C \setminus \sigma (A_\alpha)$ be fixed. Then for each~$f \in \ker (A_\alpha - \lambda)$ and each~$h \in \C^m$ we have by~\eqref{eq:Projection}
\begin{align*}
 (f, P_\lambda \cS_\zeta h) = (\cS_\zeta^* P_\lambda f, h)_{\C^m} = (\lambda - \overline{\zeta})^{-1} (\gamma_\lambda f, h)_{\C^m},
\end{align*}
which implies that $f \in \ker \gamma_\lambda$ if and only if~$f \in \{ P_\lambda \cS_\zeta h : h \in \C^m \}^\perp$. Thus
\begin{align*}
 \ker (A_\alpha - \lambda) \ominus \ker \gamma_\lambda = \left\{ P_\lambda \cS_\zeta h : h \in \C^m \right\},
\end{align*}
where~$\ominus$ denotes the orthogonal complement in~$\ker (A_\alpha - \lambda)$. Since~$\gamma_\lambda$ is an isomorphism between~$\ker (A_\alpha - \lambda) \ominus \ker \gamma_\lambda$ and~$\ran \gamma_\lambda$, the identity~\eqref{eq:nunMalLos} follows and proves~$\ran \gamma_\lambda = \ran \Res_\lambda M_{B, \alpha}$. Finally, the claimed expression for~$\dim \ker (A_\alpha - \lambda)$ is an immediate consequence of this.
\end{proof}

We remark that the proof of the preceding Proposition is inspired by more abstract considerations involving so-called Weyl or $Q$-functions and their relations to spectral properties of corresponding selfadjoint operators; see~\cite{LT77} as well as, e.g., the more recent contributions~\cite{BR14,L06}. 


\section{Recovering the spectrum from the Titchmarsh--Weyl function on the boundary}\label{sec:boundary}

In this section we study the question to what extent the spectrum of a quantum graph subject to standard or $\delta$-type matching conditions can be recovered from the knowledge of the associated Titchmarsh--Weyl function on the boundary or on parts of the boundary of~$G$. 

Let us start with three examples, which show that, in general, not the complete spectral data can be recovered from the Titchmarsh--Weyl function on~$\partial G$. The first one illustrates the situation when a cycle is present.

\begin{example}\label{ex:counterex}
In this example $G$ is the metric graph consisting of two edges $I_1 = [0, L_1]$ and $I_2 = [0, L_2]$ such that $I_2$ is a loop which is attached to the vertex~$v_2 = t (I_1)$, see Figure~\ref{fig:loop}. Suppose that the ratio of $L_1$ and $L_2$ is irrational, i.e.\ $L_1/L_2 \in \R \setminus \Q$.
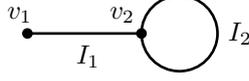
\begin{figure}[h]
\begin{center}
\begin{tikzpicture}
\pgfsetlinewidth{1pt}
\color{black}
\draw (4.5,0) circle (5mm);
\fill (4,0) circle (2pt);
\fill (2.5,0) circle (2pt);
\pgfxyline(2.5,0)(4,0)
\pgfputat{\pgfxy(3.3,-0.4)}{\pgfbox[center,base]{$I_1$}}
\pgfputat{\pgfxy(5.3,-0.1)}{\pgfbox[center,base]{$I_2$}}
\pgfputat{\pgfxy(2.4,0.2)}{\pgfbox[center,base]{$v_1$}}
\pgfputat{\pgfxy(3.75,0.2)}{\pgfbox[center,base]{$v_2$}}
\end{tikzpicture}
\end{center}
\caption{The metric graph~$G$ in Example~\ref{ex:counterex} with a loop.}
\label{fig:loop}
\end{figure}
We assume further that $\cL$ is the Laplacian on~$G$, that is, $q_1 = q_2 = 0$ identically, and that the vertex conditions are standard, that is, $\alpha_1 = \alpha_2 = 0$. Then the numbers $k^2 \pi^2/L_2^2$ with $k = 2, 4, \dots$ belong to the spectrum of the operator $A_\alpha$ in~\eqref{eq:Aalpha} and we will show that these eigenvalues do not appear as poles of the Titchmarsh--Weyl function~$M_{B, 0}$ for~$B = \partial G = \{v_1\}$. In fact, if $f = \binom{f_1}{f_2} \in \ker (A_0 - \lambda)$ with $\lambda = k^2 \pi^2/L_2^2$ for some $k \in \{2, 4, \dots \}$ then $f_2 (x) = a_2 \cos (\sqrt{\lambda} x) + b_2 \sin (\sqrt{\lambda} x)$ for appropriate $a_2, b_2 \in \C$ and, hence, $f_1' (L_1) = \partial_\nu f (v_2) + f_2' (0) - f_2' (L_2) = b_2 \sqrt{\lambda} - b_2 \sqrt{\lambda} = 0$. Together with the matching condition at~$v_1$ it follows that either $f_1 = 0$ or~$f_1$ is an eigenfunction of the Neumann Laplacian on~$[0, L_j]$ with eigenvalue~$\lambda$. In the latter case it follows $l^2 \pi^2/L_1^2 = 
\lambda = k^2 \pi^2/L_2^2$ for 
some $l \in \N$, which implies $\frac{L_1}{L_2} \in \Q$, a contradiction. Thus $f_1 = 0$ identically. Moreover, the continuity condition at~$v_2$ requires $f_2 (0) = f_2 (L_2) = f_1 (L_1) = 0$, which leads to~$a_2 = 0$. Hence
\begin{align}\label{eq:EspaceExample}
 \ker (A_0 - \lambda) = \spann \left\{ \binom{0}{f_2} \right\} \quad \text{with} \quad f_2 (x) = \sin (\sqrt{\lambda} x),~x \in I_2.
\end{align}
In particular, we have $\gamma_\lambda f = f (v_1) = 0$ for each~$f \in \ker (A_0 - \lambda)$ with~$\gamma_\lambda$ defined in~\eqref{eq:gammaLambda}. Thus Proposition~\ref{prop:entscheidend} yields $\Res_\lambda M_{B, 0} = 0$, that is, the function $M_{B, 0}$ can be continued analytically into the eigenvalues $\lambda = k^2 \pi^2/L_2^2$ with $k = 2, 4, \dots$ Consequently, these eigenvalues cannot be found with the help of~$M_{B, 0}$. 

This effect may be explained as follows. For $\mu \in \C \setminus \sigma (A_0)$ the (in the present example scalar) Titchmarsh--Weyl coefficient equals $M_{B, 0} (\mu) = f (v_1)$, where $f$ is the unique solution of $\cL f = \mu f$ in~$L^2 (G) \cap AC^2 (G)$ which is continuous on~$G$ and satisfies $\partial_\nu f (v_1) = 1$ and $\partial_\nu f (v_2) = 0$. If one considers~$\mu \in \sigma (A_0)$ the solution~$f$ is no longer unique since any eigenfunction of~$A_0$ corresponding to~$\mu$ may be added, which possibly changes the value of the solution at~$v_1$. However, in the present example all eigenfunctions corresponding to the eigenvalues $k^2 \pi^2/L_2^2$ with even~$k$ vanish at~$v_1$ so that each solution with the above properties has the same value at~$v_1$. Thus $M_{B, 0} (k^2 \pi^2/L_2^2)$ is well-defined for each even~$k$ and it turns out that this continues $M_{B, 0}$ analytically into these points. 
\end{example}

The next example shows that even for a graph without cycles the function~$M_{B, \alpha}$ may fail to detect all eigenvalues if the set of ``test vertices''~$B$ is too small.

\begin{example}\label{ex:counterex2}
In this example we consider a star graph consisting of three edges $I_1 = [0, 1]$ and $I_2 = I_3 = [0, \pi/2]$ which are glued together at one joint vertex~$v_4 =  t (I_1) = t (I_2) = t (I_3)$, see Figure~\ref{fig:star}.
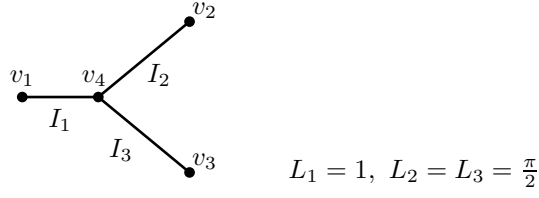
\begin{figure}[h]
\begin{center}
\begin{tikzpicture}
\pgfsetlinewidth{1pt}
\color{black}
\pgfxyline(4,0)(5.2,1);
\pgfxyline(4,0)(5.2,-1);
\fill (4,0) circle (2pt);
\fill (3,0) circle (2pt);
\fill (5.2,1) circle (2pt);
\fill (5.2,-1) circle (2pt);
\pgfxyline(3,0)(4,0)
\pgfputat{\pgfxy(3.5,-0.4)}{\pgfbox[center,base]{$I_1$}}
\pgfputat{\pgfxy(4.8,0.2)}{\pgfbox[center,base]{$I_2$}}
\pgfputat{\pgfxy(4.3,-0.8)}{\pgfbox[center,base]{$I_3$}}
\pgfputat{\pgfxy(3.0,0.2)}{\pgfbox[center,base]{$v_1$}}
\pgfputat{\pgfxy(5.4,1.1)}{\pgfbox[center,base]{$v_2$}}
\pgfputat{\pgfxy(5.4,-0.9)}{\pgfbox[center,base]{$v_3$}}
\pgfputat{\pgfxy(3.95,0.2)}{\pgfbox[center,base]{$v_4$}}
\end{tikzpicture} \hspace{10mm} $L_1 = 1,~L_2 = L_3 = \frac{\pi}{2}$
\end{center}
\caption{The star graph~$G$ in Example~\ref{ex:counterex2}.}
\label{fig:star}
\end{figure}
Consider the Laplacian~$A_0$ on~$G$ with standard matching conditions. Then $k^2 \in \sigma (A_0)$ for $k = 1, 3, \dots$ with
\begin{align*}
 \ker (A_0 - k^2) = \spann \left\{ \begin{pmatrix} 0 \\ g \\ - g \end{pmatrix} \right\}, \quad g (x) = \cos (k x),~x \in \Big[0, \frac{\pi}{2} \Big],
\end{align*}
which can be verified similar to Example~\ref{ex:counterex}.
If one chooses $B = \{v_1\}$ then each $f \in \ker (A_0 - k^2)$ vanishes on~$B$ and the same reasoning as in Example~\ref{ex:counterex} implies that $M_{B, 0}$ can be continued analytically into $\lambda = k^2$, that is, the eigenvalues~$\lambda = k^2$ with $k = 1, 3, \dots$ are invisible for~$M_{B, 0}$.
\end{example}

We remark that in Example~\ref{ex:counterex2} the situation changes fundamentally if~$B$ contains at least two boundary vertices; cf.~Corollary~\ref{cor:tree} below. In this case the matrix function~$M_{B, 0}$ is able to detect all eigenvalues of~$A_0$ with full multiplicities.

In the previous examples the choice of the specific edge lengths played a role, whereas the choice of the differential expression was restricted to the Laplacian. In the following we give another negative example, where the edge lengths are arbitrary, but a potential appears in the differential expression.

\begin{example}\label{ex:counterex3}
Let~$G$ be the graph in Example~\ref{ex:counterex} and Figure~\ref{fig:loop}, now with arbitrary edge lengths $L_1, L_2$. Let further $k_0 \in \{2, 4, \dots\}$ be fixed and consider the (constant) potentials
\begin{align*}
 q_1 = \frac{k_0^2 \pi^2}{L_2^2} - \frac{2 \pi^2}{L_1^2} \quad \text{and} \quad q_2 = 0
\end{align*}
in the differential expression~$\cL$. Then $\lambda = k_0^2 \pi^2/L_2^2$ is an eigenvalue of the operator~$A_0$. In fact, for $f = \binom{f_1}{f_2} \in \ker (A_0 - \lambda)$ one concludes as in Example~\ref{ex:counterex} that
\begin{align*}
 - f_1'' + \left( \frac{k_0^2 \pi^2}{L_2^2} - \frac{2 \pi^2}{L_1^2} \right) f_1 = \frac{k_0^2 \pi^2}{L_2^2} f_1 \quad \text{and} \quad f_1' (0) = f_1' (L_1) = 0.
\end{align*}
As the differential equation reduces to~$- f_1'' = 2 \pi^2 / L_1^2 f_1$, it follows $f_1 = 0$ identically or $2 \pi^2 / L_1^2 = l^2 \pi^2/L_1^2$ for some $l \in \N$. The latter is a contradiction, thus~$f_1 = 0$. It follows further as in Example~\ref{ex:counterex} that $\ker (A_0 - \lambda)$ is given by~\eqref{eq:EspaceExample} and that $M_{B,0}$ can be continued analytically into~$\lambda$ for~$B = \partial G = \{v_1\}$.
\end{example}

In order to formulate the main result of this section we recall the definition of the cyclomatic number of a finite graph. It is clear that this definition does not necessarily require a metric structure on the graph.

\begin{definition}
Let~$G$ be a finite, compact, connected metric graph consisting of~$r$ edges and~$s$ vertices. Then
\begin{align*}
 \cyc (G) := 1 + r - s
\end{align*}
is called the {\em cyclomatic number} of~$G$. 
\end{definition}

The number $\cyc (G)$ is the minimal number of edges that must be removed from~$G$ in order to obtain a tree (i.e. a metric graph without cycles). In that sense it counts the cycles of~$G$.

The following main result of this section states that the Titchmarsh--Weyl function for a vertex set $B \subset \partial G$ can recover all eigenvalues with sufficiently large multiplicities, depending on the size of~$\partial G \setminus B$ and the cyclomatic number of~$G$. Concerning the distinction of cases in item~(iii) we refer the reader to Remark~\ref{rem:Rem} below.

\begin{theorem}\label{thm:generalBoundary}
Let~$G$ be a finite, compact, connected metric graph and let $B \subset \partial G$ be nonempty. Moreover, let $\alpha \in \R^s$, let~$A_\alpha$ be the Schr\"odinger operator in~\eqref{eq:Aalpha}, and let~$M_{B, \alpha}$ be the Titchmarsh--Weyl function in Definition~\ref{def:TW}. Then the following assertions hold for each~$\lambda \in \R$.
\begin{enumerate}
 \item If $\lambda$ is a pole of~$M_{B, \alpha}$ then $\lambda$ is an eigenvalue of~$A_\alpha$.
 \item If $\lambda$ is an eigenvalue of~$A_\alpha$ with $\dim \ker (A_\alpha - \lambda) > 2 \cyc (G) + |\partial G| - |B| - 1$ then $\lambda$ is a pole of~$M_{B, \alpha}$.
 \item If $\cyc (G) = 0$, i.e.,~$G$ is a tree, and $B = \partial G$ then
 \begin{align*}
  \dim \ker (A_\alpha - \lambda) = \rank \Res_\lambda M_{B, \alpha}.
 \end{align*}
 Otherwise the estimate
 \begin{align*}
  \rank \Res_\lambda M_{B, \alpha} \leq \dim \ker (A_\alpha - \lambda) \leq \rank \Res_\lambda M_{B, \alpha} + 2 \cyc (G) + |\partial G| - |B| - 1
 \end{align*}
 holds.
\end{enumerate}
\end{theorem}

\begin{proof}
Throughout this proof we will assume without loss of generality that~$G$ has no vertices of degree two. If~$G$ has vertices of degree two then the two edges attached to each such vertex can be joined to one edge and this procedure, due to the matching conditions~\eqref{eq:MatchCond}, does neither change the multiplicity of an eigenvalue nor the cyclomatic number of~$G$ nor the cardinality of~$\partial G$.

Note first that the assertion~(i) is an immediate consequence of Proposition~\ref{prop:entscheidend}. Indeed, if $\lambda$ is a pole of~$M_{B, \alpha}$ then $\Res_\lambda M_{B, \alpha} \neq 0$ and it follows from Proposition~\ref{prop:entscheidend} that $\ker (A_\alpha - \lambda) \neq \{0\}$. Hence $\lambda$ is an eigenvalue of~$A_\alpha$. 

The following main part of this proof is divided into three steps and will lead to assertions~(ii) and~(iii). We elaborate the details only for the case that $\cyc (G) > 0$ or $B \neq \partial G$; the case that $\cyc (G) = 0$ and~$B = \partial G$ is completely analogous.

{\bf Step 1.} Let $\lambda \in \R$. It is the aim of this step to prove the following: If~$G$ is a finite, compact, connected metric graph with $\cyc (G) = 0$ (i.e.~$G$ is a tree) and $B \subset \partial G$ is nonempty with $B \neq \partial G$ then
\begin{align}\label{eq:bla}
 \dim \ker \gamma_{\lambda, G, B} \leq |\partial G| - |B| - 1
\end{align}
holds, where $\gamma_{\lambda, G, B} = \gamma_\lambda$ is defined as in Proposition~\ref{prop:entscheidend}. We prove this by induction over $l := |\partial G| - |B| \geq 1$.  

Let~$G$ and~$B$ as above such that~$l = 1$. We have to prove $\dim \ker \gamma_{\lambda, G, B} = 0$. Assume that there exists $f \in \ker \gamma_{\lambda, G, B}$ with $f \neq 0$, that is, $f_{j_0} \neq 0$ for some $j_0 \in \{1, \dots, r\}$. Since $f_{j_0}$ satisfies the differential equation $- f_{j_0}'' + q_{j_0} f_{j_0} = \lambda f_{j_0}$ on~$I_{j_0}$, it follows from the standard uniqueness theorem for solutions of linear ordinary differential equations that $f_{j_0} (L_{j_0}) \neq 0$ or $f_{j_0}' (L_{j_0}) \neq 0$. Thus, if $t (I_{j_0})$ is not a boundary vertex of~$G$ then the matching conditions~\eqref{eq:MatchCond} imply that there exists $j_1 \in \{1, \dots, r\}, j_1 \neq j_0$, such that $I_{j_0}$ is attached to~$t (I_{j_0})$ and $f_{j_1} \neq 0$. Indeed, if~$f_{j_0} (L_{j_0}) \neq 0$ then, due to the continuity of~$f$ at~$t (I_{j_0})$, we can choose an arbitrary edge $I_{j_1}$ attached to~$t (I_{j_0})$ which is distinct from~$I_{j_0}$; if~$f_{j_0} (L_{j_0}) = 0$ then~$f_{j_0}' (L_{j_0}) \neq 0$ 
and~\eqref{eq:MatchCond} implies that we can choose~$I_{j_1}$ different from~$I_{j_0}$ such that~$f_{j_0}$ has a non-vanishing derivative at~$t (I_{j_0})$. By the same reasoning, if~$o (I_{j_0})$ is not a boundary vertex of~$G$ then there exists $j_{-1} \in \{1, \dots, r\}$, $j_{-1} \neq j_0$, such that $I_{j_{-1}}$ is attached to~$o (I_{j_0})$ and~$f_{j_{-1}} \neq 0$. Without loss of generality we assume that $t (I_{j_{-1}}) = o (I_{j_0})$ and $o (I_{j_1}) = t (I_{j_0})$. If~$t (I_{j_1})$ is not a boundary vertex of~$G$ then another application of the described procedure yields that there exists~$j_2 \in \{1, \dots, r\}$, $j_2 \neq j_1$, such that~$I_{j_2}$ is attached to~$t (I_{j_1})$ and $f_{j_2} \neq 0$. Similarly, if~$o (I_{j_{-1}})$ is not a boundary vertex of~$G$ then there exists $j_{-2} \in \{1, \dots, r\}$, $j_{-2} \neq j_{-1}$, such that $I_{j_{-2}}$ is attached to~$o (I_{j_{-1}})$ and $f_{j_{-2}} \neq 0$. Proceeding in the same way, since~$G$ is finite and does not contain cycles, there exist 
finite 
numbers~$n, p \in \N$ such that~$o (I_{j_{-n}})$ and $t (I_{j_p})$ are 
distinct boundary vertices and such that~$f_{j_{-n}} \neq 0$ and $f_{j_p} \neq 0$. In particular, $f (o (I_{j_{-n}})) \neq 0$ or $\partial_\nu f (o (I_{j_{-n}})) \neq 0$, and $f (t (I_{j_p})) \neq 0$ or~$\partial_\nu f (t (I_{j_p}
)) \neq 0$. On the other hand, $f \in \ker \gamma_{\lambda, G, B}$ and $|\partial G| - |B| = l = 1$ imply~$f (v) = \partial_\nu f (v) = 0$ for all but one~$v \in \partial G$, a contradiction. Thus~$\ker \gamma_{\lambda, G, B} = \{0\}$, which proves~\eqref{eq:bla} for $l = |\partial G| - |B| = 1$. 

Let now $l \geq 1$ such that~\eqref{eq:bla} is satisfied whenever $|\partial G| - |B| = l$. Moreover, let $G$ and $B$ be chosen in such a way that $|\partial G| - |B| = l + 1$ and let $d := \dim \ker \gamma_{\lambda, G, B}$. We have to show that $d \leq l$. If $d = 0$ or $d = 1$ this is clear. Otherwise, let us choose a boundary vertex~$v_{i_0}$ of~$G$ which does not belong to~$B$ with corresponding edge~$I_{j_0}$; without loss of generality, $v_{i_0} = t (I_{j_0})$. Furthermore, let us choose a basis $\phi^{(1)}, \dots, \phi^{(d)}$ of~$\ker \gamma_{\lambda, G, B}$. Then $\phi^{(1)}, \dots, \phi^{(d)}$ belong to~$\ker (A_\alpha - \lambda)$ and satisfy $\phi^{(l)} (v_{i_1}) = \dots = \phi^{(l)} (v_{i_m}) = 0$, $l = 1, \dots, d$. In particular, due to the matching condition~\eqref{eq:MatchCond} we have~$\phi^{(l) \prime}_{j_0} (L_{j_0}) = \alpha_{i_0} \phi^{(l)}_{j_0} (L_{j_0})$, $l = 1, \dots, d$. Since $- \phi^{(l) \prime \prime}_{j_0} + q_{j_0} \phi^{(l)}_{j_0} = \lambda \phi^{(l)}_{j_0}$ on~$I_{j_0}$, $l = 
1, \dots, d$, it follows 
that the functions $\phi^{(1)}, \dots, \phi^{(d)}$ coincide, up to multiples, on~$I_{j_0}$. Thus, by taking linear combinations, we can 
achieve a new basis $\psi^{(1)}, \dots, \psi^{(d)}$ of~$\ker \gamma_{\lambda, G, B}$ such that $\psi_{j_0}^{(2)} = \psi_{j_0}^{(3)} = \dots = \psi_{j_0}^{(d)} = 0$ identically on~$I_{j_0}$. Let $\widetilde B := B \cup \{v_{
i_0}\}$, so that $|\partial G| - |\widetilde B| = 
l$. Then $\psi^{(2)}, \dots, \psi^{(d)}$ belong to~$\ker \gamma_{\lambda, G, \widetilde B}$ and are linearly independent. Hence
\begin{align*}
 d - 1 \leq \dim \ker \gamma_{\lambda, G, \widetilde B} \leq |\partial G| - |\widetilde B| - 1 = l - 1
\end{align*}
by the induction assumption, which leads to $d \leq l$. Thus we have proved~\eqref{eq:bla}.

{\bf Step 2.} Our aim in this step is to prove that
\begin{align}\label{eq:zwote}
 \dim \ker \gamma_{\lambda, G, B} \leq 2 \cyc (G) + |\partial G| - |B| - 1
\end{align}
holds for any finite, compact, connected metric graph~$G$ with $\cyc (G) \geq 1$ and each nonempty set~$B \subset \partial G$. We do this by induction over $\cyc (G)$. This step is inspired by the proof of~\cite[Theorem~3.2]{KP11}. 

Let first $G$ be a graph with $\cyc (G) = 1$, let $B \subset \partial G$ be nonempty, and let $d : = \dim \ker \gamma_{\lambda, G, B}$. Assume that $d \geq 2$, otherwise the claimed estimate is obvious. There exists an edge $I_{j_1}$ such that the graph obtained from~$G$ by removing the edge~$I_{j_1}$ is still connected and $o (I_{j_1})$ and $t (I_{j_1})$ are both vertices of~$G$ of degree three or larger. Let $T$ be the tree which arises from~$G$ by replacing the vertex $o (I_{j_1})$ by two vertices, one of them with degree one, containing only the left endpoint of~$I_{j_1}$, the other one with degree $\deg (o (I_{j_1})) - 1$, containing all other endpoints which were contained in the original vertex. Then $|\partial T| = |\partial G| + 1$ and $B \subset \partial T$, in particular $B \neq \partial T$. If $\dim \{f_{j_1} : f \in \ker \gamma_{\lambda, G, B}\} \leq 1$ then there exists a basis of~$\ker \gamma_{\lambda, G, B}$ with at most one function being nonzero on~$I_{j_1}$; cf.~Step~1. In this case~$d - 
1$ functions of this basis belong to~$\ker \gamma_{\lambda, T, B}$, where we impose a Neumann boundary condition at the new boundary vertex. If $\dim \{f_{j_1} : f \in \ker \gamma_{\lambda, G, B}\} = 2$ let $\chi^{(1)}, \dots, \chi^{(d)}$ be a basis 
of~$\ker \gamma_{\lambda, G, B}$, chosen in such a way that $\chi_{j_1}^{(3)} = \chi_{j_1}^{(4)} = \dots = \chi_{j_1}^{(d)} = 0$ identically on~$I_{j_1}$ and that
\begin{align*}
 & \chi_{j_1}^{(1)} (0) = 0, \quad \chi_{j_1}^{(1) \prime} (0) = 1, \\
 & \chi_{j_1}^{(2)} (0) = 1, \quad \chi_{j_1}^{(2) \prime} (0) = 0.
\end{align*}
Then the $d - 1$ functions $\chi^{(2)}, \dots, \chi^{(d)}$ belong to~$\ker \gamma_{\lambda, T, B}$. Hence in any case we obtain from~\eqref{eq:bla} 
\begin{align*}
 d - 1 \leq \dim \ker \gamma_{\lambda, T, B} \leq |\partial T| - |B| - 1 = |\partial G| - |B|,
\end{align*}
which leads to~\eqref{eq:zwote} in the case $\cyc (G) = 1$.

Assume that for some $k \geq 1$ the estimate~\eqref{eq:zwote} holds whenever $\cyc (G) = k$. Furthermore, let $G$ be a finite, compact metric graph with $\cyc (G) = k + 1$ and let~$B \subset \partial G$ be nonempty. Our aim is to show
\begin{align}\label{eq:ziel}
 d := \dim \ker \gamma_{\lambda, G, B} \leq 2 (k + 1) + |\partial G| - |B| - 1.
\end{align}
If~$d = 0$ or~$d = 1$ this is again clear. Therefore assume $d \geq 2$. 
Again, since $\cyc (G) = k + 1 \geq 2$ there exists an edge $I_{j_1}$ such that the graph~$G'$ obtained from~$G$ by removing the edge~$I_{j_1}$ is still connected and $o (I_{j_1})$ and $t (I_{j_1})$ are vertices of~$G$ of degree three or larger. Then~$G$ and~$G'$ have the same number of vertices and the same boundary, and $\cyc (G') = \cyc (G) - 1 = k \geq 1$. Since $\dim \{f_{j_1} : f \in \ker \gamma_{\lambda, G, B}\} \leq 2$ there exists a basis of~$\ker \gamma_{\lambda, G, B}$ with at most two functions being nonzero on~$I_{j_1}$, so that~$d - 2$ basis functions belong to~$\ker \gamma_{\lambda, G', B}$. Then the induction assumption implies
\begin{align*}
 d - 2 \leq \dim \ker \gamma_{\lambda, G', B} \leq 2 k + |\partial G'| - |B| - 1 = 2 k + |\partial G| - |B| - 1,
\end{align*}
which proves~\eqref{eq:ziel}. 

{\bf Step 3.} It remains to deduce from Step~2 the assertions~(ii) and~(iii) of the theorem. Indeed, it follows from~\eqref{eq:zwote} with the help of Proposition~\ref{prop:entscheidend} that
\begin{align}\label{eq:jaJa}
 \dim \ker (A_\alpha - \lambda) \leq \rank \Res_\lambda M_{B, \alpha} + 2 \cyc (G) + |\partial G| - |B| - 1.
\end{align}
Moreover, Proposition~\ref{prop:entscheidend} implies
\begin{align*}
 \rank \Res_\lambda M_B \leq \dim \ker (A_\alpha - \lambda).
\end{align*}
From the latter two estimates the assertion~(iii) follows. Finally, if~$\lambda$ is an eigenvalue of~$A_\alpha$ with~$\dim \ker (A_\alpha - \lambda) > 2 \cyc (G) + |\partial G| - |B| - 1$ then~\eqref{eq:jaJa} implies $\rank \Res_\lambda M_{B, \alpha} > 0$. Hence $\lambda$ is a pole of~$M_{B, \alpha}$. This completes the proof of the theorem.
\end{proof}

\begin{remark}\label{rem:Rem}
In Theorem~\ref{thm:generalBoundary} the case that $B = \partial G$ and $\cyc (G) = 0$ is treated separately. In fact, the larger the vertex set~$B$ gets, the more spectral information is contained in the Titchmarsh--Weyl function. However, if $G$ is a tree then the complete spectral information is contained in the Titchmarsh--Weyl function already if~$|\partial G| - |B| = 1$, see Corollary~\ref{cor:tree} below. Thus no additional information can be obtained when the missing vertex is added.
\end{remark}

The following example shows that the estimates in Theorem~\ref{thm:generalBoundary} cannot be improved in general.

\begin{example}\label{ex:counterexlast}
Consider the Laplacian subject to standard matching conditions $\alpha = 0$ on the graph in Figure~\ref{fig:twoloops} with edge lengths $L_1 = L_4 = L_5 = \pi/2, L_3 = \pi$, $L_6 = L_7 = 2 \pi$, and $L_2 = 1$. Assume that $t (I_1) = t (I_2) = o (I_3)$ and $t (I_3) = o (I_4) = o (I_5)$. 
\begin{figure}[h]
\begin{center}
\begin{tikzpicture}
\pgfsetlinewidth{1pt}
\color{black}
\draw (4.5,.7) circle (5mm);
\draw (4.5,-.7) circle (5mm);
\fill (3.5,0) circle (2pt);
\fill (2,0) circle (2pt);
\fill (4.08,0.44) circle (2pt);
\fill (4.08,-0.44) circle (2pt);
\fill (1.42,0.44) circle (2pt);
\fill (1.65,-0.32) circle (2pt);
\pgfxyline(1.42,0.44)(2,0)
\pgfxyline(1.65,-0.32)(2,0)
\pgfxyline(2,0)(3.5,0)
\pgfxyline(3.5,0)(4.08,0.44)
\pgfxyline(3.5,0)(4.08,-0.44)
\pgfputat{\pgfxy(1.85,0.3)}{\pgfbox[center,base]{$I_1$}}
\pgfputat{\pgfxy(2,-0.4)}{\pgfbox[center,base]{$I_2$}}
\pgfputat{\pgfxy(2.7,0.12)}{\pgfbox[center,base]{$I_3$}}
\pgfputat{\pgfxy(3.6,0.3)}{\pgfbox[center,base]{$I_4$}}
\pgfputat{\pgfxy(3.7,-0.5)}{\pgfbox[center,base]{$I_5$}}
\pgfputat{\pgfxy(5.27,0.7)}{\pgfbox[center,base]{$I_6$}}
\pgfputat{\pgfxy(5.25,-1)}{\pgfbox[center,base]{$I_7$}}
\pgfputat{\pgfxy(1.1,0.44)}{\pgfbox[center,base]{$v_1$}}
\pgfputat{\pgfxy(1.3,-0.4)}{\pgfbox[center,base]{$v_2$}}
\end{tikzpicture}
\end{center}
\caption{The metric graph~$G$ in Example~\ref{ex:counterexlast} with two loops.}
\label{fig:twoloops}
\end{figure}
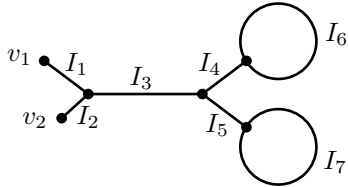
In this case the eigenspace of~$A_0$ corresponding to the eigenvalue $\lambda = 1$ is given by
\begin{align}\label{eq:lastKernel}
 \ker (A_0 - 1) = \left\{ \begin{pmatrix} a \cos \\ 0 \\ - a \sin \\ b \sin \\ (a - b) \sin \\ c \sin + b \cos \\ d \sin + (a - b) \cos \end{pmatrix} : a, b, c, d \in \C \right\}.
\end{align}
If $B = \{v_2\}$ then all these eigenfunctions vanish on~$B$, so that the operator $\gamma_1$ is trivial. In this case it follows as in the previous examples that $M_{B, 0}$ can be continued analytically into $\lambda = 1$. We have $\dim \ker (A_0 - 1) = 4 = 2 \cyc (G) + |\partial G| - |B| - 1$, which shows that the estimates in Theorem~\ref{thm:generalBoundary} are sharp in general. Similarly, if $B = \partial G = \{v_1, v_2\}$ then there is a $3$-dimensional subspace of~$\ker (A_0 - 1)$, obtained by setting $a = 0$ in~\eqref{eq:lastKernel}, which belongs to $\ker \gamma_1$, and in fact $3 = 2 \cyc (G) + |\partial G| - |B| - 1$. Thus Proposition~\ref{prop:entscheidend} yields $\rank \Res_1 M_{B, 0} = 1$ and we have $\dim \ker (A_0 - 1) = \rank \Res_1 M_{B, 0} + 2 \cyc (G) + |\partial G| - |B| - 1$. Thus the estimate in Theorem~\ref{thm:generalBoundary}~(iii) in general cannot be improved even if~$B = \partial G$. We remark that similar examples can be constructed for graphs with an arbitrary number of 
cycles.
\end{example}

\begin{remark}
Note that, although Theorem~\ref{thm:generalBoundary} above cannot be improved in general, better results can hold for special classes of graphs. For instance, let $G$ consist of a {\em cyclically connected} graph (cf.~\cite{KP11}) and one additional edge with boundary vertex attached to it. Then $\partial G$ consists of one vertex. For standard matching conditions~$\alpha = 0$ and with $B = \partial G$ Proposition~\ref{prop:entscheidend} together with the result of~\cite[Theorem~3.2~(2)]{KP11} implies that each eigenvalue~$\lambda$ with $\dim \ker (A_0 - \lambda) > \cyc (G)$ is a pole of $M_{B, 0}$ with 
\begin{align*}
 \rank \Res_\lambda M_{B, 0} \leq \dim \ker (A_0 - \lambda) \leq \rank \Res_\lambda M_{B, 0} + \cyc (G).
\end{align*}
\end{remark}

We formulate two immediate corollaries of the previous theorem. The first one is known in the case $\alpha = 0$; cf.~\cite{AK08,BW05,FY07}, where in fact stronger results are proved.

\begin{corollary}\label{cor:tree}
Let~$G$ be a finite, compact, connected metric graph with~$\cyc (G) = 0$, that is,~$G$ is a tree, and let $B \subset \partial G$ be nonempty with $|\partial G| - |B| \leq 1$. Moreover, let $\alpha \in \R^s$, let~$A_\alpha$ be the Schr\"odinger operator in~\eqref{eq:Aalpha}, and let~$M_{B, \alpha}$ be the Titchmarsh--Weyl function in Definition~\ref{def:TW}. Then the following assertions hold for each~$\lambda \in \R$.
\begin{enumerate}
	\item $\lambda$ is an eigenvalue of~$A_\alpha$ if and only if $\lambda$ is a pole of~$M_{B, \alpha}$.
	\item $\dim \ker (A_\alpha - \lambda) = \rank \Res_\lambda M_{B, \alpha}$.
\end{enumerate}
\end{corollary}

In the second corollary we emphasize the case of a metric graph with possible cycles, where $B$ contains all but at most one boundary vertices.

\begin{corollary}
Let~$G$ be a finite, compact, connected metric graph with $\cyc (G) \geq 1$ and let $B = \partial G$ be nonempty. Moreover, let $\alpha \in \R^s$, let~$A_\alpha$ be the Schr\"odinger operator in~\eqref{eq:Aalpha}, and let~$M_{B, \alpha}$ be the Titchmarsh--Weyl function in Definition~\ref{def:TW}. Then the following assertions hold for each~$\lambda \in \R$.
\begin{enumerate}
	\item If $\lambda$ is a pole of~$M_{B, \alpha}$ then $\lambda$ is an eigenvalue of~$A_\alpha$.
	\item If $\lambda$ is an eigenvalue of~$A_{\alpha}$ with $\dim \ker (A_\alpha - \lambda) > 2 \cyc (G) - 1$ then~$\lambda$ is a pole of~$M_{B, \alpha}$.
	\item The estimates
	\begin{align*}
	 \rank \Res_\lambda M_{B, \alpha} \leq \dim \ker (A_\alpha - \lambda) \leq \rank \Res_\lambda M_{B, \alpha} + 2 \cyc (G) - 1
	\end{align*}
	are satisfied.
\end{enumerate}
\end{corollary}


\section{The inverse problem for a larger vertex set}\label{sec:large}

In this section we study the situation that the knowledge of the Titchmarsh--Weyl function is available on a large vertex set~$B$ which may contain non-boundary vertices of~$G$. We start with an example which shows that even if~$B$ contains all vertices of~$G$ the Titchmarsh--Weyl function does in general not contain the full spectral information.

\begin{example}\label{ex:counterex22}
We revisit the graph in Example~\ref{ex:counterex} and Figure~\ref{fig:loop} above. We impose the same assumptions as in Example~\ref{ex:counterex}, but, in contrast to that example, we choose~$B = \{v_1, v_2\}$, that is,~$B$ consists of all vertices of the graph. Then it is obvious from the representation~\eqref{eq:EspaceExample} of the eigenspace~$\ker (A_0 - \lambda)$ for~$\lambda = k^2 \pi^2/L_2^2$ with $k = 2, 4, \dots$ that $f (v_1) = f (v_2) = 0$ for all $f \in \ker (A_0 - \lambda)$, hence the operator~$\gamma_\lambda$ defined in~\eqref{eq:gammaLambda} is trivial. Therefore Proposition~\ref{prop:entscheidend} implies $\ran \Res_\lambda M_{B, 0} = \{0\}$, that is, $M_{B, 0}$ can be continued analytically into~$\lambda$. Thus the eigenvalues~$k^2 \pi^2/L_2^2$ with ~$k = 2, 4, \dots$ are invisible for the function~$M_{B, 0}$, even though $B$ contains all vertices of~$G$.
\end{example}

In the following we give another example, where certain eigenvalues are visible but the Titchmarsh--Weyl function does not exhibit their full multiplicities.

\begin{example}\label{ex:counterex33}
Let $G$ be the metric graph without boundary which consists of two vertices $v_1, v_2$ and two edges $I_1 = I_2 = [0, \pi]$ such that $o (I_1) = t (I_2) = v_1$ and $o (I_2) = t (I_1) = v_2$, see Figure~\ref{fig:Doppelkante}.
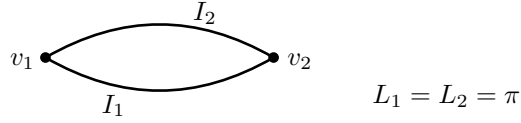
\begin{figure}[h]
\begin{center}
\begin{tikzpicture}
\pgfsetlinewidth{1pt}
\color{black}
\fill (2.5,0) circle (2pt);
\fill (5.5,0) circle (2pt);
\draw[bend left=30] (2.5,0) to (5.5, 0);
\draw[bend right=30] (2.5,0) to (5.5, 0);
\pgfputat{\pgfxy(3.4,-0.73)}{\pgfbox[center,base]{$I_1$}}
\pgfputat{\pgfxy(4.6, 0.5)}{\pgfbox[center,base]{$I_2$}}
\pgfputat{\pgfxy(2.2, -0.1)}{\pgfbox[center,base]{$v_1$}}
\pgfputat{\pgfxy(5.85,-0.1)}{\pgfbox[center,base]{$v_2$}}
\end{tikzpicture} \hspace{10mm} $L_1 = L_2 = \pi$
\end{center}
\caption{The metric graph~$G$ in Example~\ref{ex:counterex33} with a double edge connecting~$v_1$ and~$v_2$.}
\label{fig:Doppelkante}
\end{figure}
Moreover, let again $\cL$ be the Laplacian on~$G$ and consider the selfadjoint operator~$A_0$ subject to standard matching conditions at $v_1$ and $v_2$. Let us choose $B = \{v_1, v_2\}$. It is easy to check that 
\begin{align*}
 \sigma (A_0) = \left\{k^2 : k = 0, 1, 2, \dots \right\}
\end{align*}
and
\begin{align*}
 \ker (A_0 - k^2) = \spann \big\{ f^{(k)}, g^{(k)} \big\}, \quad k = 0, 1, 2, \dots,
\end{align*}
where
\begin{align*}
 f^{(k)} (x) = \binom{\sin (k x)}{\sin (k (x + \pi))} \quad \text{and} \quad g^{(k)} (x) = \binom{\cos (k x)}{\cos (k (x + \pi))}, \quad k = 0, 1, 2, \dots
\end{align*}
In particular, the eigenvalue~$0$ has multiplicity one and all other eigenvalues have multiplicity two. Since the functions $f^{(k)}$ vanish at~$v_1$ and~$v_2$, it follows that the operator~$\gamma_\lambda$ in~\eqref{eq:gammaLambda} has a one-dimensional range for $\lambda = k^2$, $k = 0, 1, 2, \dots$, and Proposition~\ref{prop:entscheidend} implies
\begin{align*}
 \rank \Res_{k^2} M_{B, 0} = 1, \quad k = 0, 1, 2, \dots
\end{align*}
Thus the poles of~$M_{B, 0}$ coincide with the eigenvalues of $A_0$, but it is not possible to read off the multiplicities of the eigenvalues from~$M_{B, 0}$.
\end{example}

In the following we provide a sufficient condition under which the Titchmarsh--Weyl function reflects all eigenvalues of~$A_\alpha$ with full multiplicities. This criterion requires the following definition.

\begin{definition}
Let~$G$ be a finite, compact metric graph. 
\begin{enumerate}
 \item The {\em core} of~$G$ is the largest subgraph of~$G$ which does not have vertices of degree one. 
 \item We call a vertex~$v$ of~$G$ {\em proper core vertex} if all edges attached to~$v$ belong to the core of~$G$.
\end{enumerate}
\end{definition}

Note that each finite, compact graph is the union of its core and a collection of disjoint rooted trees whose roots are vertices of the core (but, of course, are not proper core vertices). The core of~$G$ is empty if and only if each connected component of~$G$ is a tree. In Figure~\ref{fig:core} we illustrate an example of a graph, where the core is marked in black and the proper core vertices are drawn empty.

\begin{figure}[h]
\begin{center}
\begin{tikzpicture}
\pgfsetlinewidth{1pt}
\color{lightgray}
\draw (4.5,0) circle (5mm);
\draw (4.5,0) circle (10mm);
\pgfxyline(1.3,0)(4,0)
\pgfxyline(4.83,0.37)(6.23,1.77)
\pgfxyline(4.83,-0.37)(6.23,-1.77)
\fill (4,0) circle (2pt);
\fill (4.83,0.37) circle (2pt);
\fill (4.83,-0.37) circle (2pt);
\fill (1.3,0) circle (2pt);
\fill (6.23,1.77) circle (2pt);
\fill (6.23,-1.77) circle (2pt);
\fill (5.19,0.72) circle (2pt);
\fill (5.19,-0.72) circle (2pt);
\fill (3.5,0) circle (2pt);
\pgfxyline(2.5,0)(1.3,1);
\pgfxyline(2.5,0)(1.3,-1);
\pgfxyline(3.5,0)(2.3,-1);
\fill (2.5,0) circle (2pt);
\fill (1.3,1) circle (2pt);
\fill (1.3,-1) circle (2pt);
\fill (2.3,-1) circle (2pt);
\pgfxyline(2.5,0)(4,0)
\color{black}
\draw (4.5,0) circle (5mm);
\draw (4.5,0) circle (10mm);
\pgfxyline(3.5,0)(4,0)
\pgfxyline(4.83,0.37)(5.19,0.72)
\pgfxyline(4.83,-0.37)(5.19,-0.72)
\fill (4,0) circle (2pt);
\fill (4.83,0.37) circle (2pt);
\fill (4.83,-0.37) circle (2pt);
\fill (5.19,0.72) circle (2pt);
\fill (5.19,-0.72) circle (2pt);
\fill (3.5,0) circle (2pt);
\color{white}
\fill (4,0) circle (1pt);
\fill (4.83,0.37) circle (1pt);
\fill (4.83,-0.37) circle (1pt);
\pgfputat{\pgfxy(3.4,0.73)}{\pgfbox[center,base]{$C$}}
\end{tikzpicture}
\end{center}
\caption{A graph with its core~$C$ marked in black and the proper core vertices drawn with empty dots.}
\label{fig:core}
\end{figure}
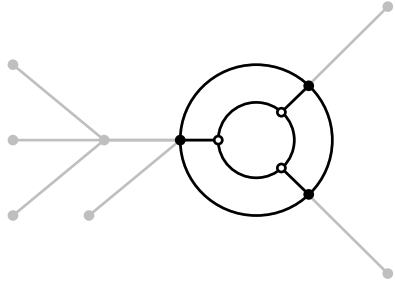

In order to formulate the next theorem a further definition is needed; cf.~also~\cite{K09}.

\begin{definition}
Let~$G$ be a finite, compact metric graph with edges~$I_1, \dots, I_r$ and let $q_j \in L^1 (I_j)$ be real-valued, $j = 1, \dots, r$. Moreover, let~$C$ be the core of~$G$. 
\begin{enumerate}
 \item A number $\lambda \in \R$ is called a {\em Dirichlet eigenvalue} of an edge $I_j$ if $\lambda$ is an eigenvalue of the selfadjoint operator $- \frac{d^2}{d x^2} + q_j$ in $L^2 (I_j)$ subject to Dirichlet boundary conditions $f_j (0) = f_j (L_j) = 0$.
 \item A number~$\lambda \in \R$ is called a {\em resonance} if there exists a cycle in~$G$ such that $\lambda$ is a Dirichlet eigenvalue of each edge in this cycle.
\end{enumerate}
\end{definition}

The following theorem is the main result of this section.

\begin{theorem}\label{thm:resonanceSingle}
Let $G$ be a finite, compact metric graph and let $B$ be a set of vertices of~$G$ which contains all boundary vertices and all proper core vertices of~$G$. Moreover, let~$\alpha \in \R^s$, let~$A_\alpha$ be the Schr\"odinger operator in~\eqref{eq:Aalpha}, and let~$M_{B, \alpha}$ be the Titchmarsh--Weyl function in Definition~\ref{def:TW}. Then for each $\lambda \in \R$ which is not a resonance the following assertions hold.
\begin{enumerate}
	\item $\lambda$ is an eigenvalue of~$A_\alpha$ if and only if~$\lambda$ is a pole of~$M_{B, \alpha}$.
	\item $\dim \ker (A_\alpha - \lambda) = \rank \Res_\lambda M_{B,\alpha}$.
\end{enumerate}
\end{theorem}

\begin{proof}
Let $\lambda \in \R$ such that~$\lambda$ is not a resonance. Moreover, let the operator~$\gamma_\lambda$ be defined as in~\eqref{eq:gammaLambda}. Assume there exists $f \in \ker \gamma_\lambda \setminus \{0\}$, that is, $f \in \ker (A_\alpha - \lambda)$, $f (v_{i_1}) = \dots = f (v_{i_m}) = 0$, where $\{v_{i_1}, \dots, v_{i_m} \} = B$, and $f_{j_0} \neq 0$ on $I_{j_0}$ for some~$j_0 \in \{1, \dots, r\}$. Then~$I_{j_0}$ belongs to the core~$C$ of~$G$. Indeed, we can decompose~$G$ into its core~$C$ and a collection of disjoint trees rooted at vertices of~$C$. If~$G'$ is one of these trees then all but at most one vertices in~$\partial G'$ belong to~$\partial G$; in particular, $f (v) = 0$ for all but at most one $v \in \partial G'$, and the same reasoning as in Step~1 of the proof of Theorem~\ref{thm:generalBoundary} yields that~$f$ vanishes identically on~$G'$ and, hence, on each of the trees. Since $f$ is continuous this implies, moreover, $f (v_i) = 0$ for $i = 1, \dots, s$; in particular, $f_{j_0}$ 
vanishes at both endpoints of~$I_{j_0}$. Since $f_{j_0} \neq 0$ and~$f_{j_0}$ satisfies~$- f_{j_0}'' + q_{j_0} f_{j_0} = \lambda f_{j_0}$ on~$I_{j_0}$ it follows that $f_{j_0}' (L_{j_0}) \neq 0$. Moreover, as~$f$ satisfies the matching condition $\partial_\nu f (t (I_{j_0})) = \alpha_{j_0} f (t (I_{j_0})) = 0$ and~$t (I_{j_0})$ is not a boundary vertex, there exists $j_1 \in \{1, \dots, r\}$ such that~$j_1 \neq j_0$, $I_{j_1}$ is adjacent to~$I_{j_0}$ with joint vertex $t (I_{j_0})$ and belongs to~$C$, and $f_{j_1} \neq 0$. Analogously, there exists $j_{-1}$ such that $I_{j_{-1}}$ is adjacent to~$I_{j_0}$ with joint vertex $o (I_{j_0})$ and $f_{j_{-1}} \neq 0$. By repeating this procedure similar to Step~1 in the proof of Theorem~\ref{thm:generalBoundary}, since the core of~$G$ consists of finitely many edges we obtain~$n, p \in \N$ such that $I_{j_{-n}}, \dots, I_{j_p}$ or a part of this path form a cycle and $f_j \neq 0$ for all $j \in \{j_{-n}, \dots, j_p\}$. Thus $\lambda$ is a Dirichlet eigenvalue of 
all edges in that cycle, which means that~$\lambda$ is a resonance, a contradiction. It follows~$\ker \gamma_\lambda = \{0\}$ and Proposition~\ref{prop:entscheidend} immediately leads to the assertions of the theorem.
\end{proof}

Let us illustrate the choice of the vertex set~$B$ in the previous theorem by an example.

\begin{example}\label{ex:Kurasov}
The graph~$G$ in Figure~\ref{fig:Kurasov} consisting of one cycle and three boundary edges attached to it was considered in~\cite{K09}, where it was shown that the Titchmarsh--Weyl function for $B = \partial G = \{v_1, v_2, v_3\}$ determines the potentials on the edges of~$G$ uniquely, provided there are no resonances. As this graph does not possess proper core vertices, under the same conditions Theorem~\ref{thm:resonanceSingle} yields the (weaker) result that on this graph the Titchmarsh--Weyl function determines all eigenvalues and its multiplicities.
\begin{figure}[h]
\begin{center}
\begin{tikzpicture}
\pgfsetlinewidth{1pt}
\color{black}
\draw (4.5,0) circle (5mm);
\pgfxyline(3,0)(4,0)
\pgfxyline(4.83,0.37)(5.53,1.07)
\pgfxyline(4.83,-0.37)(5.53,-1.07)
\fill (4,0) circle (2pt);
\fill (4.83,0.37) circle (2pt);
\fill (4.83,-0.37) circle (2pt);
\fill (3,0) circle (2pt);
\fill (5.53,1.07) circle (2pt);
\fill (5.53,-1.07) circle (2pt);
\pgfputat{\pgfxy(2.8,0.2)}{\pgfbox[center,base]{$v_1$}}
\pgfputat{\pgfxy(3.8,0.2)}{\pgfbox[center,base]{$v_4$}}
\pgfputat{\pgfxy(5.8,1.2)}{\pgfbox[center,base]{$v_2$}}
\pgfputat{\pgfxy(4.8,0.6)}{\pgfbox[center,base]{$v_5$}}
\pgfputat{\pgfxy(5.8,-1)}{\pgfbox[center,base]{$v_3$}}
\pgfputat{\pgfxy(5.2,-0.4)}{\pgfbox[center,base]{$v_6$}}
\end{tikzpicture}
\end{center}
\caption{A graph whose core consists of one cycle.}
\label{fig:Kurasov}
\end{figure}
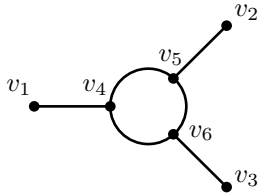
\end{example}

Observe that for the Laplacian on~$G$ without potentials the Dirichlet eigenvalues on an edge $I_j$ are given by $\frac{k^2 \pi^2}{L_j^2}$ for $k = 1, 2, \dots$ Thus in this case Theorem~\ref{thm:resonanceSingle} implies the following corollary. 

\begin{corollary}\label{cor:LaplaceEV}
Let $G$ be a finite, compact metric graph and let $B$ contain all boundary vertices and all proper core vertices of~$G$. Moreover, let $\cL$ be the Laplacian on $G$, i.e., $q_j = 0, j = 1, \dots, r$. Let~$\alpha \in \R^s$, let~$A_\alpha$ be the Schr\"odinger operator in~\eqref{eq:Aalpha}, and let~$M_{B, \alpha}$ be the Titchmarsh--Weyl function in Definition~\ref{def:TW}. Assume that each cycle in~$G$ contains two edges with rationally independent edge lengths. Then the assertions of Theorem~\ref{thm:resonanceSingle} hold for each $\lambda \in \R$.
\end{corollary}

\section*{Appendix: Analytic operator functions}

In this appendix we shortly recall basic facts and notions on analytic matrix and operator functions as used in the main part of this paper. For further details we refer the reader to~\cite{DS71}.

Let $\cH$ be a complex Hilbert space, which may be finite- or infinite-dimensional. Moreover, let $G \subset \C$ be a nonempty, open set and let $R (z)$ be a bounded, everywhere defined operator in~$\cH$ (a matrix if $\dim \cH < \infty$) for each $z \in G$. Assume that the operator function $R$ is analytic, that is, it can be represented locally by a power series which converges with respect to the operator topology. We say that some~$\lambda \in \C$ is a pole of~$R$ of order~$n$ if there exists an open neighborhood~$B$ of~$\lambda$ in~$\C$ such that $\overline{B} \setminus \{\lambda\} \subset G$, 
\begin{align*}
 \lim_{\mu \to \lambda} (\mu - \lambda)^n R (\mu)~\text{exists~and~is~nontrivial}, \quad \text{and} \quad \lim_{\mu \to \lambda} (\mu - \lambda)^{n + 1} R (\mu) = 0
\end{align*}
in the operator topology. The function~$R$ is called meromorphic if~$\C \setminus G$ consists of isolated points which are poles; this may include poles of order zero, i.e., removable singularities. For any~$\lambda \in \C$ the residue of~$R$ at~$\lambda$ is the bounded linear operator in~$\cH$ given by
\begin{align*}
 \Res_\lambda R = \frac{1}{2 \pi i} \int_\Gamma R (\mu) d \mu,
\end{align*}
where~$\Gamma$ is any closed Jordan curve in~$G$ which surrounds $\lambda$ but no other point in~$\C \setminus G$. If $\lambda \in G$ or $\lambda$ is a removable singularity of~$R$ then $\Res_\lambda R = 0$; if~$R$ is meromorphic then $\lambda$ is a pole of~$R$ of positive order if and only if~$\Res_\lambda R$ is nontrivial. Alternatively, the residue can be defined as the first coefficient of negative order in the Laurent series expansion of~$R$ centered at~$\lambda$. If~$\lambda$ is a pole of order one of~$R$ then $\Res_\lambda R = \lim_{\mu \to \lambda} (\mu - \lambda) R (\mu)$.

As a standard example, let~$A$ be a (not necessarily bounded) selfadjoint operator in~$\cH$ with a compact resolvent; in particular, the spectrum $\sigma (A)$ consists of isolated eigenvalues with finite multiplicities. In that case the function $\mu \mapsto R (\mu) = (A - \mu)^{-1}$, $\mu \in \C \setminus \sigma (A)$, is meromorphic with poles of order one precisely at the eigenvalues of~$A$. Moreover, for each $\lambda \in \R$ the relation
\begin{align}\label{eq:projection}\tag{A.1}
 P_\lambda = - \Res_\lambda R
\end{align}
holds, where $P_\lambda$ denotes the orthogonal projection in~$\cH$ onto~$\ker (A - \lambda)$.


\begin{thebibliography}{99}

\bibitem{AK08} S.~Avdonin and P.~Kurasov, {\it Inverse problems for quantum trees}, Inverse Probl.\ Imaging 2 (2008), 1--21.

\bibitem{BL10} J.~Behrndt and A.~Luger, {\it On the number of negative eigenvalues of the Laplacian on a metric graph}, J.\ Phys.\ A 43 (2010), 474006, 11 pp.

\bibitem{BR14} J.~Behrndt and J.~Rohleder, {\it Spectral analysis of selfadjoint elliptic differential operators, Dirichlet-to-Neumann maps, and abstract Weyl functions}, arXiv:1404.0922.

\bibitem{BKW04} G.~Berkolaiko, J.\,P.~Keating, and B.~Winn, {\it No quantum ergodicity for star graphs}, Comm.\ Math.\ Phys.\ 250 (2004), 259--285.

\bibitem{BK13} G.~Berkolaiko and P.~Kuchment, Introduction to quantum graphs, Mathematical Surveys and Monographs 186, American Mathematical Society, Providence, RI, 2013.

\bibitem{BW05} B.\,M.~Brown and R.~Weikard, {\it A Borg--Levinson theorem for trees}, Proc.\ R.\ Soc.\ Lond.\ Ser.~A Math.\ Phys.\ Eng.\ Sci.\ 461 (2005), 3231--3243.

\bibitem{C14} Y.~Colin de Verdi\`ere, {\it Semi-classical measures on quantum graphs and the Gau\ss~map of the determinant manifold}, Ann.\ Henri Poincar\'e 16 (2015), 347--364.

\bibitem{CW09} S.~Currie and B.\,A.~Watson, {\it The M-matrix inverse problem for the Sturm--Liouville equation on graphs}, Proc.\ Roy.\ Soc.\ Edinburgh Sect.\ A 139 (2009), 775--796. 

\bibitem{CIM98}  E.\,B.~Curtis, D.~Ingerman, and J.\,A.~Morrow, {\it Circular planar graphs and resistor networks}, Linear Algebra Appl.\ 283 (1998), 115--150.

\bibitem{D13}  E.\,B.~Davies, {\it An inverse spectral theorem}, J.\ Operator Theory 69 (2013), 195--208.

\bibitem{DS71} N.~Dunford and J.\,T.~Schwartz, Linear Operators I, Wiley, New York, 1958.

\bibitem{EK11} M.~Enerb\"ack and P.~Kurasov, {\it Aharonov-Bohm ring touching a quantum wire: how to model it and to solve the inverse problem}, Rep.\ Math.\ Phys.\ 68 (2011), 271--287.

\bibitem{E14} P.~Exner, {\it Resonances in quantum networks and their generalizations}, in Nonlinear Phenomena in Complex Systems: From Nano to Macro Scale, Springer, 2014, pp.~159--178.

\bibitem{EL10} P.~Exner and J.~Lipovsk\'y, {\it Resonances from perturbations of quantum graphs with rationally related edges}, J.\ Phys.\ A 43 (2010), 105301, 21 pp.

\bibitem{FY07} G.~Freiling and V.~Yurko, {\it Inverse problems for Sturm--Liouville operators
on noncompact trees}, Results Math.\ 50 (2007), 195--212.

\bibitem{GSS13} S.~Gnutzmann, H.~Schanz, and U.~Smilansky, {\it Topological resonances in scattering on networks (graphs)}, Phys.\ Rev.\ Lett.\ 110 (2013), 094101, 5 pp.

\bibitem{GS01}  B.~Gutkin and U.~Smilansky, {\it Can one hear the shape of a graph?}, J.\ Phys.\ A 34 (2001), 6061--6068.

\bibitem{KP11} I.~Kac and V.~Pivovarchik, {\it On multiplicity of a quantum graph spectrum}, J.\ Phys.\ A 44 (2011), 105301, 14 pp.

\bibitem{KS99} V.~Kostrykin and R.~Schrader, {\it Kirchhoff's rule for quantum wires}, J.\ Phys.~A 32 (1999), 595--630.

\bibitem{KS06} V.~Kostrykin and R.~Schrader, {\it Laplacians on metric graphs: eigenvalues, resolvents and semigroups}, Quantum graphs and their applications, 201--225, Contemp.\ Math., 415, Amer.\ Math.\ Soc., Providence, RI, 2006.

\bibitem{KS04} T.~Kottos and H.~Schanz, {\it Statistical properties of resonance width for open quantum systems},  Waves Random Media 14 (2004), S91--S105.

\bibitem{K08} P.~Kuchment, {\it Quantum graphs: an introduction and a brief survey}, Analysis on graphs and its applications, 291--312, Proc.\ Sympos.\ Pure Math., 77, Amer.\ Math.\ Soc., Providence, RI, 2008.

\bibitem{K09} P.~Kurasov, {\it On the inverse problem for quantum graphs with one cycle}, Acta Physica Polonica~A 116 (2009), 765--771.

\bibitem{K11} P.~Kurasov, {\it Inverse problems for quantum graphs: recent developments and perspectives}, Acta Physica Polonica~A 120 (2011), A-132--A-141.

\bibitem{KN05} P.~Kurasov and M.~Nowaczyk, {\it Inverse spectral problem for quantum graphs}, J.\ Phys.\ A 38 (2005), 4901--4915.

\bibitem{LT77} H.~Langer and B.~Textorius, {\it On generalized resolvents and $Q$-functions of symmetric linear relations (subspaces) in Hilbert space}, Pacific J.\ Math.\  72 (1977), 135--165.

\bibitem{L06} A.~Luger, {\it A characterization of generalized poles of generalized Nevanlinna functions}, Math.\ Nachr.~279 (2006), 891--910. 

\bibitem{N07} M.~Nowaczyk, {\it Inverse spectral problem for quantum graphs with rationally dependent edges}, Oper.\ Theory Adv.\ Appl.\ 174 (2007), 105--116.

\bibitem{P07}  V.~Pivovarchik, {\it Inverse problem for the Sturm--Liouville equation on a star-shaped graph}, Math.\ Nachr.\ 280 (2007), 1595--1619.

\bibitem{P12} O.~Post, Spectral Analysis on Graph-like Spaces, Springer Lecture Notes 2039, 2012.

\bibitem{SK03} H.~Schanz and T.~Kottos, {\it Scars on quantum networks ignore the Lyapunov exponent}, Phys.\ Rev.\ Lett.\ 90:234101 (2003).

\bibitem{TM01} C.~Texier and G.~Montambaux, {\it Scattering theory on graphs}, J.\ Phys.\ A 34 (2001), 10307--10326.

\bibitem{WS13} D.~Waltner and U.~Smilansky, {\it Scattering from a ring graph --- a simple model for the study of resonances}, Acta Physica Polonica~A 124 (2013), 1087--1090.

\end{thebibliography}
\end{document}